\documentclass[12pt]{amsart}

\usepackage[text={6in,8in},centering,letterpaper]{geometry}
\usepackage[dvipsnames]{xcolor}
\usepackage{esvect}
\usepackage{graphicx}      
\usepackage{tikz}
\usepackage{amsmath,amssymb, amsfonts,amsthm}
\usepackage{subfig}
\usepackage{graphicx}
\usepackage{pgfplots}
\usepackage{tikz}
\usepackage{tikz-cd}
\newcommand{\R}{\mathbb{R}}

\newcommand{\im}{\mathrm{im\ }}

\renewcommand{\ker}{\mathrm{ker\ }}

\newcommand{\Z}{\mathbb{Z}}

\usepackage{microtype}

\makeatletter
\def\blfootnote{\xdef\@thefnmark{}\@footnotetext}
\makeatother

\newcommand{\xc}[1]{\vspace{.1cm}

\noindent {\em #1} }

\newcommand{\mab}[1]{\vspace{.1cm}

\noindent {\textbf{#1 }}} 

\newtheorem{definition}{Definition}
\newtheorem{theorem}{Theorem}
\newtheorem{proposition}[theorem]{Proposition}
\newtheorem{lemma}{Lemma}

\usepackage[style= numeric-comp,hyperref=true, doi=false,url=false,
            isbn=false,
            firstinits=true, sorting = none, 
            block=none, backend=bibtex,maxnames=99]{biblatex}
\renewbibmacro{in:}{} 
\bibliography{globstab.bib}

\title{Dynamic Global Feedback Stabilization: why do the twist?}

\date{}
\author{Mohamed-Ali Belabbas and Jehyung Ko} 

\thanks{M.-A. Belabbas and Jehyung Ko are with  Coordinated Science Laboratory, University of Illinois, Urbana-Champaign. Email: \texttt{belabbas@illinois.edu}}

\begin{document}
\maketitle

\begin{abstract}              We investigate  global dynamic feedback stabilization from a topological viewpoint. In particular, we consider the   general case of dynamic feedback systems, whereby the total  space (which includes the state space of the system and of the controller) is a fibre bundle, and derive conditions on the topology of the bundle that are necessary for various notions of global stabilization to hold. This point of view highlight the importance of distinguishing trivial bundles  and twisted bundles  in the study of global dynamic feedback stabilization, as we show that dynamic feedback defined on a twisted bundle can stabilize systems that dynamic feedback on trivial bundles cannot. 
\end{abstract}

\section{Introduction}

A well-known necessary condition for global {\em state feedback} stabilization of a control system is that the state-space of the system be contractible (see below for a precise definition). This condition simply reflects the fact that the flow of a globally stable system induces a contraction to the (necessarily unique) equilibrium of the system. 
The case of {\em dynamic} feedback stabilization, which we address in this paper, is far more involved. In addition to the obvious fact that one can design the controller space and dynamics, one can also design {\em topology} of the state-space of the {\em closed-loop} system, i.e., how the controller space is ``attached'' to the state-space of the system. This second aspect, which is very often overlooked, will be shown here to strongly impact what can be achieved by dynamic feedback.

One of the main reasons why the design of the topology of the global system is less explored stems from the following fact: Let $M$ be the state-space of the system and $U$ the controller space; a dynamic feedback control system is often simply written  as
\begin{align}\label{eq:dyn1}
\begin{cases}
\dot{x}=f(x,u)\\
\dot{u}=g(x,u)
\end{cases}
\end{align}
and the goal of the designer is to derive an appropriate $g$ for the control task at hand. However, this {\em implicitly} assumes that the global state-space is $E=M \times U$, and is thus a {\em trivial} fibre bundle, where trivial specifically refers to the fact that $E$ is (globally) the product of $M$ and $U$. In order to describe the dynamics over a non-trivial bundle (or twisted bundle), one needs to either embed it in a higher dimensional state-space, or to cover the base space with charts and write dynamics of the type~\eqref{eq:dyn1} for each chart while verifying some compatibility conditions on charts' overlaps. We refer the reader to~\cite{bredon2013topology,boothby1986introduction} for more details.

There is scant literature on the topology of global dynamic feedback. But among related works, we mention the seminal works~\cite{brockett1976nonlinear}, where one can find more extended discussion on the fibre bundle structure of feedback control systems,~\cite{brockett1983asymptotic} where a topological obstruction to local state-feedback stabilization was exhibited. A similar condition was shown independently in~\cite{krasnoselskij1984geometrical}. The case of global state-feedback stabilization was treated in~\cite{bhat2000topological}, where an explicit proof of the fact that state-feedback global stabilization defines a contraction can be found. Obstructions to local stabilization to a submanifold, extending Brockett's result, have been derived in~\cite{mansouri2007local,mansouri2010topological}. The use of topological methods, more precisely homology, has also been fruitful in robotics~\cite{leve2020homological}. Finally, we mention the recent work~\cite{kvalheim2022necessary} where issues of state feedback stabilization to compact subsets  while avoiding regions of the state-space are addressed, and~\cite{baryshnikov2023topological}, where the author studies the complexity of the subset of $M$ over which a stabilizing state feedback has to be discontinuous.  For a broader view topological ideas in control, we  refer to~\cite{leve2024}.

Throughout this paper, the state-space, control space and total space of a feedback system are finite-dimensional smooth manifolds, and CW complexes.
\section{On dynamic feedback and fibre bundles}

We gather the definitions that are needed to state our results, some being standard in the literature and other specific to our problem. 

For our purpose, the total space $E$ for dynamic feedback control system is a fibre bundle defined as follows
\begin{definition}A smooth fibre bundle is the data $(E,p,M,U)$ where $E, M$ and $U$ are smooth manifolds, $p:E \to M$ is a smooth projection map onto $M$ so that for every $x \in M$, there is an open neighborhood $V_x \ni x$ and a diffeomorphism $\varphi_x:p^{-1}(V_x) \to V_x \times U$ so that the diagram
\begin{equation*}
\begin{tikzcd}[row sep=huge]
p^{-1}(V_x) \subseteq E \arrow[r,"\varphi_x"] \arrow[d,swap,"p"] &
V_x \times U  \arrow[dl,"p_1"] 
\\
V_x  & & 
\end{tikzcd}
\end{equation*}
commutes, where $p_1$ is the projection onto the first component. We also denote the bundle as  $ p: E \to M$ with fibre $U$.
\end{definition}
\begin{figure}
\centering
\begin{tikzpicture}[scale=0.75]
\begin{axis}[
    hide axis,
    view={40}{40},
    zmin=-.05,
    zmax=.05,
]
\addplot3 [
    surf, shader=faceted interp,
    point meta=x,
    colormap/cool,
    samples=40,
    samples y=3,
    z buffer=sort,
    domain=-180:180,
    y domain=-0.1:0.1
] (
    {2*cos(x)},
    {2*sin(x)},
    {0.125*y});

\addplot3 [
    samples=50,
    domain=-135:35, 
    samples y=0,
    thick
] (
    {2*cos(x)},
    {2*sin(x)},
    {0});
\addplot3 [
    samples=50,
    domain=50:210, 
    samples y=0,
    thick
] (
    {2*cos(x)},
    {2*sin(x)},
    {0});
\addplot3 [
    samples=50,
    domain=-.1:.1, 
    samples y=3,
    very thick,
    red
] (
    {2*cos(-97)},
    {2*sin(-97)},
    {0.125*y});
\node[label={220:{\color{red} \Large x}},circle,fill,red,inner sep=1pt] at (axis cs:-1.8,-.14) {};

\end{axis}
\begin{scope}[xshift=7.cm]
\begin{axis}[
    hide axis,
    view={40}{40},
    zmin=-.45,
    zmax=.45,
]
\addplot3 [
    surf, shader=faceted interp,
    point meta=x,
    colormap/cool,
    samples=40,
    samples y=3,
    z buffer=sort,
    domain=-180:180,
    y domain=-0.2:0.2
] (
    {(1+0.5*y*cos(x/2)))*cos(x)},
    {(1+0.5*y*cos(x/2)))*sin(x)},
    {0.5*y*sin(x/2)});

\addplot3 [
    samples=50,
    domain=-142:205, 
    samples y=0,
    thick
] (
    {cos(x)},
    {sin(x)},
    {0});
\addplot3 [
    samples=50,
    domain=-.2:.2,
    samples y=3,
    very thick,
    red
] (
    {(1+0.5*y*cos(-97/2)))*cos(-97)},
    {(1+0.5*y*cos(-97/2)))*sin(-97)},
    {0.5*y*sin(-97/2)});
\node[label={270:{\color{red} \Large x}},circle,fill,red,inner sep=1pt] at (axis cs:-.91,-.03) {};
\end{axis}
\end{scope}
\end{tikzpicture}
    \caption{We illustrate the topology of dynamic feedback for $M=S^1$ (unit circle, in black) and $U=[-1,1]$ (unit interval). {\it Left:} The total space is the trivial bundle $E=M \times U$. {\it Right:} The total space is the Mobius strip, which is a bundle over $S^1$ with fibre $[-1,1]$. 
Stabilizing a point $x \in M$ (red dot) entail stabilizing the fibre $p^{-1}(x)$ in $E$ (red segment).}
    \label{fig:mobiusstri}
\end{figure}
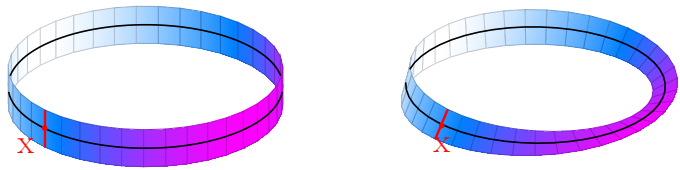
The functions $\varphi_x$ are called local trivialization of the fibre bundle. If we can take $V_x = M$, then  $E$ is diffeomorphic to $M \times U$ via $\varphi_x$ and the bundle is called {\em trivial}.

Equipped with this definition, we can now define the kind of general dynamic feedback control system we address in this note. 

\begin{definition}[Dynamic feedback system] Let $M$ be a smooth manifold, and $f(\cdot,\cdot): M \times U \to TM$ be a smooth $u$-dependent vector field so that $f(x,u) \in T_xM$ for all $u \in U$. Then, a {\em dynamic feedback} for $f(x,u)$ is given by a fibre bundle $p: E \to M$  with fibre $U$ and a vector field $f_{cl}$ on $E$ which can be  written as
\begin{align}\label{eq:dyn2}
\begin{cases}
\dot{x}=f(x,u)\\
\dot{u}=g(x,u)
\end{cases}
\end{align}
using the trivializations $\varphi_x$.
\end{definition}

Consider the system $\dot x = f(x(t))$ on a manifold $M$ and let $x \in M$ be an equilibrium of this system. We recall the following standard definition.
\begin{definition}\label{def:stab}
An equilibrium $x^*\in M$ is said to be Lyapunov stable if for all neighborhoods $V_1$ of $x^*$, there exists a neighborhood $V_2$ of $x^*$ so that if $x_0 \in V_2$, $e^{tf}x_0 \in V_1$, where $e^{tf}x_0$ is the solution at time $t$  of the system initialized at $x_0$. Furthermore, $x^*$ is said to be asymptotically stable if it is Lyapunov stable and there exists a neighborhood $V_3$ so that  $\lim_{t \to \infty} x_0 = x$ for all $x_0 \in V_3$. Finally, $x^*$ is globally asymptotically stable if it is asymptotically stable and $V_3$ can be taken  equal to $M$.
\end{definition}
The above definition can easily be extended to stabilization to a submanifold of $M$.

\mab{Dynamic feedback stabilization} When dealing with {\em dynamic} feedback stabilization on $p:E \to M$, one needs to refine the notion of stability. Indeed, using Definition~\ref{def:stab} to stabilize at a point $z^* \in E$ so that $p(z^*)=x^*$, though it achieves the goal of stabilization of $x^*$, can be too strong a requirement as it forces the controller to reach a particular state as well. Said otherwise, one needs to account for the fact that only stabilization at a particular $x^* \in M$ can matter, and the state of the controller could be irrelevant when the desired state for the system is reached.  Hence, when dealing with dynamic feedback, two notions of global stabilization arise.

\xc{Strong global stabilization:} A control system on $p:E \to M$ is strongly globally stable if $z^* \in E$ is a globally stable equilibrium.
\xc{Weak global stabilization:} A control system on $p:E \to M$ is weakly globally stable if $U^*:=p^{-1}(x^*) \in E$ is a globally stable equilibrium.

We will also account for the fact that very often, while global stabilization is impossible, removing a single point from the state-space renders it possible (for a deeper exploration of this phenomenon, we refer to~\cite{baryshnikov2023topological}. We will refer to this as $1$-point almost global stabilization. We emphasize that for $1$-point almost global stabilization, we require convergence from all points of $E$  except from $z_1 \in E$; this provides a much stronger form of stability than the usual almost global stability on $M$. Indeed, in the latter case, we require convergence to $x^*$ from all initial states except from a point $x_1 \in M$; this  translates into requiring convergence on $E$ save from $p^{-1}(x_1)$, which is much larger than $z_1$, unless the control space $U$ is of finite cardinality. 

We gather these notions in the definition

\begin{definition}
    Consider the control system $\dot x = f(x,u)$ on $M$. Then, a dynamic feedback closed loop system on $p: E \to M$
    \begin{enumerate} 
    \item strongly stabilizes $x^*$ if it stabilizes some $z^* \in p^{-1}(x^*)$.
    \item weakly stabilizes $x^*$ if it stabilizes the fibre  $p^{-1}(x^*)$.
    \item {\em 1-point almost globally} weakly/strongly stabilizes $x^*$ if there exist one point $z_1 \in E$ so that $x^*$ is weakly/strongly globally stabilized over $E -\{z_1\}$.
    \end{enumerate}
\end{definition}
See Figure~\ref{fig:mobiusstri} for an illustration.

\section{Statement of the main results.}

Our main results are

\begin{theorem}\label{th:main1}
Consider a control system $\dot x=f(x,u)$ with state-space $M$ and control space $U$. Then, $x^* \in M$ is weakly/strongly globally stabilizable by dynamic feedback only if $M$ is contractible.	
\end{theorem}
This result states that dynamic feedback, even with the choice of topology of the global state-space, cannot globally stabilize an equilibrium, in a weak or strong sense, unless the state-space of the system is itself contractible. For the second result, we assume that $M$ is moreover closed, which in light of Lemmas~\ref{lem:contrbundletrivial} and~\ref{lem:closednotcontr} is the only interesting case.

\begin{theorem}\label{th:main2}
Consider a control system $\dot x=f(x,u)$ with state-space $M$, a closed finite-dimensional smooth manifold, and control space $U$, a finite-dimensional smooth manifold. Then, $x^* \in M$ is weakly/strongly $1$-point globally stabilizable by dynamic feedback only if the closed-loop system has a total space $E$ which is a {\em nontrivial} bundle.
\end{theorem}
This result highlights a major gain to be had by considering non-trivial bundles for dynamic feedback: if one allows for the removal of one point in the {\em total space}---which, as we mentioned earlier, is a weaker requirement than removing one point in $M$---then  a dynamic feedback designed on a trivial bundle $E=M \times U$, cannot $1$-point almost globally stabilize $x_0$, but no such obstruction arises for dynamic feedback designed on a {\em twisted} bundle.

To conclude this section, we  state two well-known facts from topology, whose proofs can be found, e.g., in~\cite{hatcher2005algebraic}, that  better establish the reach of our results. We first recall the definition of {\em deformation retraction}: given a space $M$ and a subspace $A \subseteq M$, a deformation retraction of $M$ to $A$ is a continuous map $r:[0,1] \times M \to M$ so that $r(0,x)=x$ for all $x \in M$, $r(1,x)\in A$ for all $x \in M$ and $r(t,x)=x$ for all $x \in A$,$t \in [0,1]$. The space $M$ is called {\em contractible} if it deformation retracts onto one of its points. 

The first fact is the following.
\begin{lemma}\label{lem:contrbundletrivial}
    If $M$ is contractible, all bundles on $M$ are trivial.
\end{lemma}
From this result, we conclude that when designing dynamic feedback on contractible spaces, we do not have access to the topology of the total space as a design parameter, beyond the choice of the topology of the control space $U$. Said otherwise, if $M$ is contractible, all dynamic feedback will evolve on $M \times U$.

Second, it follows from definitions that any contractible space is simply-connected, but the converse is not true. The simplest example being the sphere $S^2$ which is neither simply-connected but contractible space. Generally, we have the following result

\begin{lemma}\label{lem:closednotcontr}
Assume that $M$ is a connected compact smooth manifold without boundary. Then $M$ is not contractible.
\label{lemma:compactandcontractible}
\end{lemma}
From this result, we conclude, e.g., that when designing dynamic feedback controllers for robotic or satellite systems, whose state spaces contain copies of $SO(3)$ or  $SE(3)$, or formation systems, whose state-space contains a copy of $\mathbb{C}P(n)$, the topology of the total space does matter.

\section{Background on algebraic topology.}
To make the paper as self-contained as possible, we gather some results and notions from algebraic topology that are needed in this paper. We refer the reader to~\cite{may1999concise, hatcher2005algebraic} for more details.
 This section can be safely skipped by anyone with a working knowledge of the topic. 
 Given groups $A_1, A_2,\ldots$, an exact sequence is given by a sequence of homomorphisms $\ell_1:A_1 \to A_2$, $\ell_2:A_2 \to A_3, \ldots$ with the property that
$$
\im \ell_{i} =\ker l_{i+1}
$$ for all $i$. An exact sequence is denoted as
$$
A_1 \xrightarrow{\ell_1} A_2\xrightarrow{\ell_2} A_2 \xrightarrow{\ell_3} \cdots
$$
As an example, note that the exact sequence
\begin{equation}\label{eq:exactiso}
0  \xrightarrow{p } A \xrightarrow{q} B \xrightarrow{r} 0  	
\end{equation}
implies that $A$ and $B$ are isomorphic; indeed, by exactness, $\im p = \ker q$ but since the domain of $p$ is $\{0\}$, its image is also $0$ and $q$ thus has a trivial kernel.  Further, because $r$ maps into the group with single element zero, its kernel is $B$ and by exactness, $\im q = \ker r = B$. Hence, $q$ is both injective and surjective and thus an isomorphism. A less trivial example is the following
\begin{lemma}[Five-lemma]
	Consider the commutative diagram
	$$
\begin{tikzcd}
A_1 \arrow{r}{\ell_1}\arrow[swap]{d}{p_{1}} & A_2 \arrow{r}{\ell_2}\arrow[swap]{d}{p_{2}} & A_3\arrow[swap]{d}{p_{3}}\arrow{r}{\ell_3} &  A_4\arrow[swap]{d}{p_4}\arrow{r}{\ell_4} &A_5\arrow[swap]{d}{p_{5}}\\
B_1 \arrow{r}{m_1} & B_2\arrow{r}{m_2} & B_3 \arrow{r}{m_3} & B_4 \arrow{r}{m_4}& B_5
\end{tikzcd}
$$
where the rows are exact sequences.
Assume that $p_2$ and $p_4$ are isomorphisms, $p_1$ is surjective and $p_5$ is injective. Then $p_3$ is an isomorphism.  
\end{lemma}

We denote by $H_i(M;A)$ the $i$th homology group of $M$ with coefficients in $A$. We will make use of the following fact below.
\begin{lemma}\label{lem:onhn}
Let $M$ be a connected, smooth closed manifold of dimension $n$. Then if $M$ is orientable, $H_n(M;\Z)=\Z$ and it is otherwise zero. \end{lemma}
 
We denote by $\pi_i(M)$ the $i$th homotopy group of $M$, $i \geq 1$, and $\pi_0(M)$ is the set of connected components of $M$, which by abuse of terminology is referred to as $0$th homotopy group of $M$. We will make use of the following result relating homotopy group and homology groups:
\begin{theorem}[Hurewicz]\label{th:hurewicz}
Let $M$ be a path connected space. Then for all $k \geq 1$, there exists a homomorphism $h_*:\pi_k(M) \to H_k(M;\Z)$. Furthermore, for $k \geq 2$, if $\pi_{i}(M)=0$ for all $1\leq i\leq k-1$, then $h_*:\pi_k(M) \to H_k(M;\Z)$ is an isomorphism.
\end{theorem}
The maps $h_*$ are referred to as Hurewicz maps. See~\cite{hatcher2005algebraic} for a proof. 

Let $x^* \in M$; the pair $(M,x^*)$ is  called a {\em pointed space}. When $x^*$ is clear from the context, we will simply write $M$ for the pointed space $(M,x^*)$. We denote by $PM$ the  {\em path space} of $(M,x^*)$, which is defined as
$$PM:=\{\gamma:[0,1] \to M \mbox{ with }\gamma(1) = x^*\},
$$
where $\gamma$ is continuous. It is equipped with the compact-open topology~\cite{hirsch2012differential}. The {\em loop space} $\Omega M$ of $(M,x^*)$ is the subspace of $PM$ defined as
$$\Omega M:=\{\gamma:[0,1] \to M \mbox{ with }\gamma(0)=\gamma(1) = x^*\}.
$$
One can show that $\pi_i(\Omega M)=\pi_{i+1}(M)$. 

Consider the map $H:[0,1] \times PM \to PM: (s,\gamma(t)) \mapsto \gamma(s+t(1-s))$. For $s=0$, it is the identity on $PM$ and for $s=1$ is it is constant path $x^*$. This map is in fact a deformation retraction of $PM$ onto $x^*$, showing that $PM$ is {\em contractible}.   The loop space and path space of $(M,x^*)$ yield the so-called {\em path space fibration} 
\begin{equation}\label{eq:defpathfibration}
\Omega M \xrightarrow{i} PM \xrightarrow{q} M
\end{equation}
where the first map is the inclusion map discussed above, and $q(\gamma(t))= \gamma(0)$.

We now introduce exact sequences that will be used below.
A fibration $U \to E \to M$ obeys the following long exact sequence of homotopy groups
\begin{multline}\label{eq:homles}
\cdots \longrightarrow \pi_n(U) \longrightarrow \pi_n(E) \longrightarrow \pi_n(M)\longrightarrow \pi_{n-1}(U) \longrightarrow  \\
\cdots \longrightarrow \pi_1(M) \longrightarrow \pi_0(U) \longrightarrow \pi_0(E) \longrightarrow \pi_0(M)	
\end{multline}
Note that this sequence also holds for the fibration $\Omega M \to PM \to M$.

Let $A,B \subseteq M$ be subspaces so that $\operatorname{int}A \cup \operatorname{int} B =M$, where $\operatorname{int} A$ denotes the interior of $A$. The {\em Mayer-Vietoris} sequence is the long exact sequence in homology given by
\begin{multline}\label{eq:mv}
\cdots\to  H_{i+1}(M) \to H_i(A \cap B) \to H_i(A) \oplus H_i(B) \to H_i(M) \to \cdots\\  \to H_0(A) \oplus H_0(B) \to H_0(M) \to 0. 	
\end{multline}
This sequence will prove useful when dealing with $1$-point stabilization.

When dealing with trivial bundles, the K\"unneth formula, which describes the homology group of a product of spaces, will be useful. For a principal ideal domain $G$, and CW complexes $M_1$ and $M_2$, it states that the following sequence is exact:
\begin{multline}\label{eq:kunneth}
0 \xrightarrow{\ell_1}\bigoplus_{i+j=k} H_i(M_1;G) \otimes H_j(M_2;G) \xrightarrow{\ell_2} H_k(M_1 \times M_2;G) \xrightarrow{\ell_3}\\ \bigoplus_{i+j=k-1}\operatorname{Tor}(H_i(M_1;G),H_j(M_2;G)) \xrightarrow{\ell_4} 0 	
\end{multline}

We conclude by stating the Whitehead theorem, which we will rely on several times. To do so, we first introduce the following definition: a map $f: A \to B$ between topological spaces is a {\em weak homotopy equivalence} if it induces an isomorphism on all homotopy groups; namely if $f_*:\pi_k(A) \to \pi_k(B)$ is an isomorphism for all $k \geq 0$. The theorem states
\begin{theorem}[Whitehead's Theorem]\label{th:whitehead}
	A weak homotopy equivalence between CW complexes is a homotopy equivalence.
\end{theorem}

\section{Proof of the results}

The following Proposition states the relationship between global asymptotic stabilization and dynamic feedback control.
\begin{proposition}\label{prop:dfcandgas}
Consider the control system $\dot x = f(x,u)$ on $M$ with control  $u \in U$. If a dynamic feedback with state-space $p: E \to M$   globally (resp.~$1$-point almost globally) stabilizes  $x^* \in M$,   then there exists a deformation retraction of $E$ onto $p^{-1}(x^*)$ (resp.~a deformation retraction of $E-\{z_1\}$ onto $p^{-1}(x^*)$) for some point $z_1 \in E$. 
\end{proposition}
The proof of this proposition is immediate; it suffices to time-rescale the flow $e^{tf(x,u)}:[0,\infty] \times E \to E$ to a flow $\varphi:[0,1] \times E \to E$ to  obtain the contraction sought after. We omit the details.

Equipped this Proposition~\eqref{prop:dfcandgas}, we now restate our main Theorems in topological terms. The equivalent of Theorem~\ref{th:main1} is as follows:
\begin{theorem}\label{theorem:MaintheoremA}
Let $M, U$ be  smooth, finite dimensional manifolds. Let $p:E \to M$ be a fibre bundle with fibre $U$. Then $E$ deformation retracts onto a point or a fibre only if $M$ is contractible.
\end{theorem}

We now restate Theorem~\ref{th:main2} in topological terms.
\begin{theorem}\label{theorem:MaintheoremB}
Let $M, U$ be  smooth, finite dimensional manifolds with $M$ closed. Let $p:E \to M$ be a fibre bundle with fibre $U$. Then, there exists a $z_1 \in E$ such that $E^1:=E-\{z_1\}$ deformation retracts onto a point $z^*\in E$ or a fibre $p^{-1}(x^*)$, for some $z^*, x^*$, only if $p:E \to M$ is a nontrivial bundle. 
\end{theorem}
We recall that $M,U$ are assumed to be CW complexes. 
\subsection{Proof of Theorem~\ref{theorem:MaintheoremA}}

We will treat the case of a deformation retraction onto a fibre and onto a point separately, as they require  different approaches. We start with deformation onto a fibre.

\mab{Obstruction to weak global stabilization. } 

We show here that if $p:E \to M$ deformation retracts onto one of its fibres, then $M$ is contractible. 
Let $x^* \in M$ and $U^* = p^{-1}(x^*)$ be a fibre onto which $E$ deformation retracts. Then, the inclusion map $i: U^* \to E$ is a homotopy equivalence.

Now, recall the long exact sequence of homotopy groups of a fibration
$$ \cdots \xrightarrow{} \pi_k(U) \xrightarrow{i_*}  \pi_k(E) \xrightarrow{p_*} \pi_k(M) \xrightarrow{\delta_*} \pi_{k-1}(U) \longrightarrow \cdots
$$

Since $i$ is a homotopy equivalence, $i_*$ is an isomorphism; thus, by exactness of the sequence, $\ker p_* = \im i_* = \pi_k(E)$.
Relying on exactness at $\pi_k(M)$, we get that $\ker \delta_*=\im p_*$; but  $\im p_*$ is $\{0\}$ because its kernel equals its domain by the above, hence $\delta_*$ is injective. It then follows that
$$\im \delta_* \simeq \pi_k(M).$$
Appealing to the exactness of the  sequence one more time, we have that $\im \delta_* = \ker i_*$. However, since $i_*$ is an isomorphism, its kernel is  trivial. We conclude
\begin{equation}\label{eq:pinzero}
\pi_k(M) = 0 \mbox{ for all } k \geq 1.	
\end{equation}
Since $M$ is connected, $\pi_0(x^*)=\pi_0(M)$.
Now consider the inclusion map
$$ j: \{x^*\} \to M.$$
 Owing to~\eqref{eq:pinzero}, we have  that $j_*: \pi_k(x^*) \to \pi_k (M)$ is an isomorphism for all $k \geq 0$.  Hence, $j$ is a weak homotopy equivalence.
From Whitehead's Theorem (Theorem~\ref{th:whitehead}), we conclude that $j$ is in fact a homotopy equivalence and thus $M$ is homotopic to a point. This proves the first statement. 

\mab{Obstruction to strong global stabilization.}
We now show that if the total space of the bundle $p:E \to M$ is contractible, then $M$ is contractible. The proof relies on a couple of Lemmas.

\begin{lemma}\label{lem:simplyconnected}
Let $p:E \to M$ be a fibre bundle with contractible total space. Then $M$ is simply-connected.
\end{lemma}
\begin{proof}
	From the long exact sequence~\eqref{eq:homles}, we have that
	$$ \pi_1(E) \xrightarrow{\ell_1} \pi_1(M) \xrightarrow{\ell_2} \pi_0(U) \xrightarrow{\ell_3} \pi_0(E)$$
	is an exact sequence. Because $E$ and $U$ are connected, $\pi_0(E)$ and $\pi_0(U)$ are singletons and thus $\ell_3$ is necessarily bijective and $\ell_2$ maps onto the unique element of $\pi_0(U)$.  Because $E$ is contractible, $\pi_1(E)=0$. and thus $\im \ell_1 =0$. By exactness, we conclude that  $\ell_2$ has a trivial kernel. Hence, we have that $\pi_1(M)$ contains only one element, and thus $\pi_1(M)=0$, which proves the result. 
\end{proof}
The next lemma relates the loop space of $M$ to the fibre $U$: 
\begin{lemma}\label{lem:homotopyloopfiber}
Let $p:E \to M$ be a fibre bundle with contractible total space. Then $\Omega M$ is homotopy equivalent to $U$.
\end{lemma}
This proof follows a standard construction; we include it for completeness.
\begin{proof}
 Fix  $z^* \in p^{-1}(x^*)$ and consider the pointed space $(M,x^*)$. Since $E$ is contractible, there exists a deformation retraction $r:[0,1] \times E \to E$ onto $z^*$. Then, $r_z(t):t \mapsto r(t,z)$ is the path traced by $z$ as the deformation retraction is applied. It is path in $E$ but after applying $p$, it yields a path in $M$ which ends at $x^*$ by construction. Hence, we have the map $m_1:E \to PM: z \mapsto p(r_z(t))$. In turn, when restricted to $U=p^{-1}(x^*)$, $m_1$ induces a map $m_2: U \to \Omega B$.  These maps yield the following commutative diagram.
$$
\begin{tikzcd}
 U \arrow{r}{i_1}\arrow[swap]{dd}{m_2} & E\arrow{dr}{p}\arrow[swap]{dd}{m_1} & \\
 & & M\\
 \Omega M\arrow{r}{i_1} & PM\arrow[swap]{ur}{q}  &
\end{tikzcd}
$$
where $i_1$ and $i_2$ are inclusion maps, and $q$ is the evaluation map at $t=0$, defined in~\eqref{eq:defpathfibration}. Now taking the long exact sequences in homotopy for the fibration $U \to E \to M$ and $\Omega M \to PM \to M$, we get a commutative diagram

$$
\begin{tikzcd}
\pi_{k+1}(E)\arrow[swap]{d}{m_{1*}}\arrow{r}{p_*} &  \pi_{k+1}(M)\arrow[swap]{d}{id_{*}}\arrow{r}{} & \pi_k(U) \arrow{r}{i_*}\arrow[swap]{d}{m_{2*}} & \pi_k(E)\arrow[swap]{d}{m_{1*}}\arrow{r}{p_*} &  \pi_k(M)\arrow[swap]{d}{id_{*}}\\
 \pi_{k+1}(PM) \arrow{r}{q_*} & \pi_{k+1}(M) \arrow{r}{} & \pi_k(\Omega M)\arrow{r}{i_{1*}} & \pi_k(PM) \arrow{r}{q_*} & \pi_k(M) 
\end{tikzcd}
$$
Since $PM$ and $E$ are contractible, $m_{1*}$ is an isomorphism; so is $id_*$. From the five lemma, we conclude that $m_{2*}$ is an isomorphism and hence $m_2:U \to \Omega M$ is a weak homotopy equivalence.

Since $M$ is a CW-complex, it is known~\cite{milnor1959spaces} that $\Omega M$ has the homotopy type of a CW complex as well. We thus conclude from Whitehead theorem that $U$ is homotopy equivalent to $\Omega M$. 	 
\end{proof}

We now proceed by contradiction and assume that such a dynamic feedback exists, and denote its state-space by the fibre bundle $p:E \to M$. As above,  it would induce a deformation retraction of the total space $E$ onto a point $z^* \in p^{-1}(x^*)$. From Lemma~\ref{lem:homotopyloopfiber}, this implies that the control space $U$ is homotopy equivalent to the loop space $\Omega M$ of $(M,x^*)$. But by assumption, the control space $U$ is finite-dimensional, say of dimension $r$, and thus its homology groups eventually vanish, i.e,  $H_{r+i}(U;F)=0$ for all $i \geq 1$.  Since $U$ is homotopy equivalent to $\Omega M$, it implies that 
\begin{equation}\label{eq:homofiberloop} H_i(\Omega M;F) \simeq H_i(U;F) = 0 \mbox{ for all }i \geq r+1.
 	 \end{equation}

We now show that~\eqref{eq:homofiberloop} can hold {\em only if} $M$ is contractible. To this end, we rely on the following result of Serre~\cite[Proposition 11]{serre1951homologie}:

\begin{proposition}[Serre] Let $F$ be a field and  $M$ be simply connected. If for some $n \geq 2$, $H_n(M;F) \neq 0$ and $H_i(M;F) = 0$ for all $i > n$, then for all integers $k \geq 0$, there exists $0 < j <n$ so that $H_{k+j}(\Omega M;F) \neq 0$. 
\end{proposition}
We now have two cases:
\begin{description}
	\item[\it Case 1: $H_n(M;F) \neq 0$ for some $n \geq 2$ and field $F$:] 

In this case, the homology of $\Omega M$ does not vanish, which contradicts~\eqref{eq:homofiberloop} and concludes the proof.

\item[\it Case 2: $H_n(M;F) = 0$ for all $n \geq 2$ and all fields $F$:] From Lemma~\ref{lem:simplyconnected}, we have
\begin{equation}\label{eq:h1p10} H_1(M;\Z)=\pi_1(M)=0.
\end{equation}

Now, it follows from the Universal Coefficients Theorem that if for all $n \geq 1$, it holds that $H_n(M;F)=0$ for all fields $F$  then in fact \begin{equation}\label{eq:hn0}
H_n(M,\Z)=0.
 \end{equation}
See, for example,~\cite[Corollary 3A.7]{hatcher2005algebraic}. 

From Hurewicz Theorem~\ref{th:hurewicz} and~\eqref{eq:h1p10}, we have that  $h_*:\pi_2(M) \to H_2(M;\Z)$ is an isomorphism, whence using~\eqref{eq:hn0}, $\pi_2(M)=0$. Since $\pi_2(M)=0$, using the same arguments, we conclude that $h_*:\pi_3(M) \to H_3(M;\Z)=0$ is an isomorphism. Iterating in this fashion,  we obtain that $\pi_k(M)=0$ for all $k \geq 1$. 

Hence, as above, the inclusion map $j:\{x^*\} \to M$ is thus a weak homotopy equivalence and, by Whitehead Theorem, we conclude that $M$ is contractible.
\end{description}

\subsection{Proof of Theorem~\ref{theorem:MaintheoremB}}

The proof relies on two lemmas which we prove here. To illustrate the first one, assume that the homology groups of $M$ and $U$ are all freely generated. In this case, the K\"unneth formula for $E= M \times U$ reduces to
\begin{equation}\label{eq:kunlm1}
H_n(E)=\bigoplus_{i+j=n} H_i(M) \otimes H_j(U)
\end{equation}
Because $U$ is  connected, $H_0(U) = \Z.$ Since  $\dim M = n$ and $M$ is closed, we have $H_n(M) = \Z$ since we assumed the homology groups were freely generated. Thus 
\begin{equation}\label{eq:lemmpre1}
H_n(M) \otimes H_0(U) = \Z.
\end{equation}
 Furthermore, since $M$ is connected, $H_0(M;\Z)=\Z$ and we have 
\begin{equation}\label{eq:lemmpre2}
H_0(M) \otimes H_n(U)=\Z \otimes H_n(U) = H_n(U).
\end{equation}
 Putting these together,  we get from~\eqref{eq:kunlm1} that
\begin{equation}\label{eq:inlemma}
	H_n(E) = H_n(U) \oplus \Z \oplus  \bigoplus_{i+j=n,i\neq0,j\neq 0}H_i(M) \otimes H_j(U),
\end{equation}
showing that $H_n(E)$ is non-zero and not equal to $H_n(U)$. More generally, we have the following Lemma:

\begin{lemma}\label{lem:hnuorient}
Let $p: E \to M$ be a trivial bundle with fibre $U$ and $\dim M = n$. If $M$ is orientable, $H_{n}(E) \neq 0 $ and $H_{n}(E) \neq H_n(U) $; otherwise $H_{n-1}(E)\neq 0$ and $H_{n-1}(E)\neq H_{n-1}(U)$.
\end{lemma}
\begin{proof}
We deal with both cases separately.
\xc{$M$ is orientable:} Since $M$ is closed, it holds that $H_n(M)=\Z$. The K\"unneth sequence says that 
\begin{equation}\label{eq:kunneth2}
0 \xrightarrow{\ell_1}\underbrace{\bigoplus_{i+j=n} H_i(M) \otimes H_j(U)}_{A} \xrightarrow{\ell_2} H_n(E) \xrightarrow{\ell_3}\\ B\xrightarrow{\ell_4} 0 	
\end{equation}
is exact, where $B$ is given in~\eqref{eq:kunneth}.  Now by exactness, $0= \im \ell_1 = \ker \ell_2$ and thus $\ell_2$ is injective. But from~\eqref{eq:lemmpre1} and~\eqref{eq:lemmpre2}, we know that $A=H_n(U) \oplus \Z \oplus \cdots$  thus $H_n(E) \neq 0$ and $H_n(E) \neq H_n(U)$.
\xc{$M$ is not orientable:} 
As above, we have $H_{n-1}(U) \otimes H_{0}(M) = H_{n-1}(U)$ and $H_0(U) \otimes H_{n-1}(M) = H_{n-1}(M)$. Now consider the K\"unneth sequence of degree  $n-1$, and label the first term $A$ and last term $B$ as in~\eqref{eq:kunneth2}. The term $A$ is in that case
$$
A = H_{n-1}(U) \oplus H_{n-1}(M) \oplus \cdots 
$$
If $M$ not orientable, since it is a connected closed manifold we necessarily have that $H_{n-1}(M) \neq 0$ (see, e.g., \cite[Corollary 3.26]{hatcher2005algebraic}). Using the fact that $\ell_2$ is injective, as was shown above, we conclude that $H_{n-1}(E) \neq H_{n-1}(U)$ and $H_{n-1}(E) \neq 0$.

\end{proof}

The next result shows that the homology of  $E^1$ is very similar to the one of $E$.

\begin{lemma}\label{lem:eqhomominus}
Let $p:E \to M$ be a fibre-bundle with $\dim E = m$. Then, for $1\leq k\leq m-2$, we have $H_k(E)=H_k(E^1).$
\end{lemma}
\begin{proof}
Recall that the local homology at $z_1$ of $E$ is defined as $H_i(E,E-\{z_1\})=H_i(E,E^1)$.
Now, a standard argument (excising the complement $B^c$ of an open ball $B$ containing $z_1$ using  the long exact sequence for the  pair $H_i(B,S^{m-1})$, which is equal to  $H_i(E,E^1)$ by the excision isomorphism)  shows that the local homology is 
$$
H_i(E,E^1)= 0 \mbox{ for } 1 \leq i <m.
$$
Now using the previous result in the long exact sequence for the pair $(E,E^1)$, we obtain

$$
H_{i+1}(E,E^1)=0 \xrightarrow{} H_i(E^1) \xrightarrow{}H_i(E) \xrightarrow{}0=H_i(E,E^1),
$$
for $1 \leq i \leq m-2$. Using the argument below~\eqref{eq:exactiso} yields the result. 
\end{proof}

As in the proof of Theorem~\ref{theorem:MaintheoremA}, we prove the case of weak and strong stabilization separately.

\mab{Obstruction to strong $1$-point global stabilization}

We prove that $E^1$ deformation retracts onto a point  only if $E$  is a nontrivial bundle. To this end, assume for contradiction that $E$ is trivial and  $E^1$ is contractible. In this case, $H_k(E^1)=0$ for $k \geq 1$.

We now consider three cases for the control space $U$. Either $\dim U \geq 2$, or $\dim U =1$, which can be split into $U=\R$ or $U=S^1$ since $U$ is  a connected manifold.  Recall that $\dim M = n$.

\xc{Case $\dim U \geq 2$:} Then, $m = \dim E \geq n+2$. Owing to Lemma~\ref{lem:eqhomominus}, we  have $H_{i}(E)=H_{i}(E^1)=0$ for $i=n-1,n$. But Lemma~\ref{lem:hnuorient} states that either $H_n(E)$ or $H_{n+1}(E)$is non-zero if $E$ is trivial, which yields the contradiction sought.

\xc{Case $U = \R$:} In this case,  $E=M \times \R$, which is homotopy equivalent to $M$. Hence, $H_{k}(E)\simeq H_k(M)$ and in particular $H_{m}(E)=0$. 

Let $B$ be an open ball containing $z_1$.	Recall the Mayer-Vietoris long exact sequence, where we write $E = E^1 \cup B$:
\begin{multline}\label{eq:mayviet2}
 \cdots \longrightarrow  H_{m}(E)  \longrightarrow H_{m-1}(E^1 \cap B)  \longrightarrow H_{m-1}(E^1) \oplus H_{m-1}(B)  
\longrightarrow \\ H_{m-1}(E)  \longrightarrow H_{m-2}(E^1 \cap B)  \longrightarrow \cdots
\end{multline}

To proceed, we record the following two observations. First, $E^1 \cap B \simeq S^{m-1}$, and $H_k(S^{m-1})$ is nonzero only for $k=m-1$ and $k=0$, in which cases it is $\Z$. Second, $H_k(B)=0$ for all $k >0$.

From the Mayer-Vietoris sequence~\eqref{eq:mayviet2} starting at $H_m(E)=0$, and using the above observations, we have
\begin{equation*}
0  \xrightarrow{\ell_1} \Z  \xrightarrow{\ell_2} H_{m-1}(E^1) \xrightarrow{\ell_3} H_{m-1}(E)  \xrightarrow{} 0
\end{equation*}
Since the sequence is exact, we conclude that $\ell_2$ is injective and thus $H_{m-1}(E^1) \neq 0$ and thus $E^1$ is not contractible.

\xc{Case $U = S^1$:} We have $E= M \times S^1$. From the K\"unneth exact sequence in degree 1, we obtain
\begin{multline*}
0 \rightarrow H_0(M)\otimes H_1(S^1)\oplus H_1(M)\otimes H_0(S^1) \to H_1(M \times S^1) \\ \to \operatorname{Tor}(H_0(M),H_0(S^1))\to 0
\end{multline*}
Since $H_0(M)=H_0(S^1)=\Z$, the torsion term vanish, and we get
\begin{equation}\label{eq:skunnethS1}0 \to \Z \oplus H_1(M) \to H_1(E) \to 0.	
\end{equation}
This shows that $H_1(E)\neq 0$. But $H_1(E^1)$  vanishes by assumption, so Lemma~\ref{lem:eqhomominus} yields the contradiction sought.

\mab{Obstruction to weak $1$-point global stabilization.}

We prove that $E^1$ deformation retracts onto a fibre  only if $E$  is a nontrivial bundle. 
Proceeding by contradiction, we assume that $E$ is trivial and $E^1$ deformation retracts onto a fibre. In particular, this implies that $H_i(E^1)\simeq H_i(U)$ for $i \geq 0$. We deal with the same three cases for the fibre $U$ as above:

\xc{Case $\dim U \geq 2$:} As above, we have   $H_{n}(E)=H_{n-1}(E)=0$. But Lemma~\ref{lem:hnuorient} states that if $E$ is trivial, either $H_n(E) \neq H_n(U)$ or $H_{n-1}(E)\neq H_{n-1}(U)$; this provides a contradicton.

\xc{Case $U = \R$:} In this case, deformation retraction of $E^1$ onto a fibre is equivalent to deformation retraction onto a point, and the proof is the same as above.
\xc{Case $U = S^1$:} If $M$ is moreover non-orientable, then from Lemma~\ref{lem:hnuorient} we have $H_{n-1}(E) \neq H_{n-1}(U)$ but  Lemma~\ref{lem:eqhomominus} says that $H_{n-1}(E)=H_{n-1}(E^1)$. Hence $E^1$ cannot deformation retract onto $U$.

Now assume that $M$ is orientable. We claim that $M$ is necessarily a {\em homology sphere}.
\begin{lemma}\label{lem:Mishomosphere}
	If $E^1=E-\{z_1\}$ is deformation retracts to $S^1$,  with $E=M \times S^1$ and $M$ orientable, then $M$ is an integral homology sphere. 
\end{lemma}
\begin{proof}
	Recall $\dim M = n$. From Lemma~\ref{lem:eqhomominus}, we have that $H_i(E^1)=H_i(E)$ for $0 \leq i \leq n-1$. The K\"unneth sequence for $E$ in degree 1 is given in~\eqref{eq:skunnethS1}, and shows that $$H_1(E) =\Z \oplus H_1(M).$$
	But since we assumed that $E^1$ deformation retracts onto $U=S^1$, then $H_1(E^1)=H_1(U) = \Z$. Whence using  Lemma~\ref{lem:eqhomominus}, $H_1(E) = \Z$ and thus $H_1(M)=0$.
	
	Next, the K\"unneth sequence for $E$ in degree 2 is
\begin{multline}
0 \xrightarrow{} [H_0(S^1) \otimes H_2(M)] \oplus [H_1(S^1) \otimes H_1(M)] \oplus [H_2(S^1) \otimes H_0(M)] \xrightarrow{} H_2(E) \\\xrightarrow{} 	\operatorname{Tor}(H_1(M),H_0(S^1)) \oplus \operatorname{Tor}(H_0(M),H_1(S^1))
\end{multline}
Since $H_1(M)=0$, and all other groups appearing in the torsion terms are $\Z$, the torsion terms vanish. The sequence then becomes
$$
0 \rightarrow H_2(M) \to H_2(E) \to 0.
$$
The above says that $H_2(M)=H_2(E)$ and using Lemma~\ref{lem:eqhomominus}, we get $H_2(M)=H_2(E^1)$ and by assumption $H_2(E^1)=0$. Now assuming by induction that $H_{i-1}(M)=0$ for $1 \leq i \leq k < n-1$, we have that 
$$\bigoplus_{i+j=k+1} H_i(M) \otimes H_j(U) = H_{k+1}(M) \otimes H_0(U) = H_{k+1}(M).
$$
and
$$\bigoplus_{i+j=k}\operatorname{Tor}(H_i(M),H_j(S^1)) = 0.
$$
The K\"unneth sequence in degree $k+1$ then yields that $H_{k+1}(M)=H_{k+1}(E)$ and Lemma~\ref{lem:eqhomominus} that $H_{k+1}(E)=H_{k+1}(E^1)$. This proves the claim that  $H_i(M)=0$ for $1 \leq i \leq n-1$. Since $H_n(M) = \Z$ by assumption, then $M$ is an integral homology sphere.

\end{proof}

The Mayer-Vietoris sequence~\eqref{eq:mayviet2}  yields

$$
H_{n+1}(E) \xrightarrow{} H_n(S^n) \xrightarrow{} H_n(E^1) \xrightarrow{\ell_1} H_n(E) \xrightarrow{\ell_2} H_{n-1}(S^n)
$$
Assume that $n > 1$; then $H_n(E^1)=H_n(U)=0$ and $H_{n-1}(S^n)=0$, yielding the exact sequence
$$
H_{n+1}(E) \xrightarrow{} \Z \xrightarrow{} 0 \xrightarrow{\ell_1} H_n(E) \xrightarrow{\ell_2} 0.
$$
By exactness,  $H_n(E)=0$. 
However, using Lemma~\ref{lem:Mishomosphere}, the K\"unneth sequence in degree $n$ for $E$ reduces to
$$ 0 \xrightarrow{} H_n(M) \xrightarrow{} H_n(E) \xrightarrow{}0
$$
and thus $H_n(E)=H_n(M)=\Z$, which contradicts our previous conclusion.

It remains to deal with the case $n=1$. In this case, since $M$ is a closed manifold, we have $M=S^1$ and $E^1 = S^1 \times S^1 -\{z_1\}$. It is easy to see that $E^1 \simeq S^1 \lor S^1$, a wedge   of two circles, which clearly does not retract onto $S^1$. This concludes the proof.

\section{Acknowledgments}
The first author thanks Alexandra Kjuchukova and Frederick Leve for helpful discussions and comments on the manuscript.
\printbibliography
\end{document}